\documentclass[12pt]{amsart}
\usepackage{amscd,amssymb,hyperref}
\usepackage[PostScript=dvips]{diagrams}
\usepackage{amsfonts}
\usepackage[all]{xy}

\newtheorem{thm}{Theorem}[section]
\newtheorem{lem}[thm]{Lemma}
\newtheorem{cor}[thm]{Corollary}
\newtheorem{prop}[thm]{Proposition}

\def\square{\vbox{
      \hrule height 0.4pt
      \hbox{\vrule width 0.4pt height 5.5pt \kern 5.5pt \vrule width 0.4pt}
      \hrule height 0.4pt}}

\def\id{\mathrm{id}}

\def\ch\mathrm{ch}

\def\sk{\mathrm{sk}}
\def\RP{\mathbb{R}\mathrm{P}}
\def\CP{\mathbb{C}\mathrm{P}}

\long\def\symbolfootnote[#1]#2{\begingroup%
\def\thefootnote{\fnsymbol{footnote}}\footnote[#1]{#2}\endgroup}

\newcommand{\Z}{\mathbb{Z}}

\newcommand{\R}{\ensuremath{\mathbb{R}}}

\numberwithin{equation}{section}

\title{On the metastable homotopy of mod $2$ Moore spaces}

\author[R.~Mikhailov]{Roman~Mikhailov}
\address{Chebyshev Laboratory, St. Petersburg State University, 14th Line, 29b,
Saint Petersburg, 199178 Russia}
\address{St. Petersburg Department of Steklov Mathematical Institute
}
\email{rmikhailov@mail.ru}

\author{J. Wu }
\address{Department of Mathematics, National University of Singapore, 10 Lower Kent Ridge Road, Singapore 119076} \email{matwuj@nus.edu.sg}
\urladdr{www.math.nus.edu.sg/\~{}matwujie}

\thanks{The main result (Theorem~\ref{theorem1.1}) is supported by Russian Scientific Foundation, grant N 14-21-00035. The last author is also partially supported by the Singapore Ministry of Education research grant (AcRF Tier 1 WBS No. R-146-000-190-112)
and a grant (No. 11329101) of NSFC of China.}
\begin{document}

\begin{abstract}
In this article, we study the exponents of metastable homotopy of
mod $2$ Moore spaces. Our result gives that the double loop space
of $4n$-dimensional mod $2$ Moore spaces has a multiplicative
exponent $4$ below the range of $4$ times the connectivity. As a
consequence, the homotopy groups of $4n$-dimensional mod $2$ Moore
spaces have an exponent $4$ below the range of $4$ times the
connectivity.

\end{abstract}

\maketitle


\section{Introduction}

The metastable homotopy has been studied by various people with fruitful results~\cite{Barratt54, Barratt-Mahowald, Baues, GH,James, Mahowald, Mahowald1, Morisugi, Tipple} since the early 1950s. The descriptions on the lower metastable homotopy groups of the Moore spaces given by M. G. Barratt~\cite{Barratt54} in 1954 leaded to computational results announced in~\cite{Tipple}. In this article, we consider the exponents of the metastable homotopy  groups of mod $2$ Moore spaces.

Let $P^n(2)=\Sigma^{n-2}\RP^2$ be the $n$-dimensional mod $2$ Moore space with $n\geq3$. It is well known that $P^n(2)$ has a suspension exponent of $4$, that is, the degree $4$ map $[4]\colon P^n(2)\to P^n(2)$ is null homotopic. By the classical result of M. G. Barratt~\cite{Barratt60}, the metastable homotopy of $P^n(2)$ has an exponent dividing $8$. This leads to a natural question whether  the metastable homotopy of $P^n(2)$ has an exponent of $4$. The answer to this question is negative for the cases $n\equiv 2, 3\mod{4}$ because there are $\Z/8$-summands occurring in the lower metastable homotopy groups according to~\cite{Cohen-Wu, Mukai, Tipple}. The purpose of this article is to give an affirmed answer to the above question for the case $n\equiv0\mod{4}$ with $n>5$.

Our answer to the question is actually given by showing that  the double loop space $\Omega^2P^n(2)$ has a multiplicative exponent $4$ below the range roughly $4$ times the connectivity in the case $n\equiv0\mod{4}$ with $n>5$. More explicitly, our main result is as follows.

\begin{thm}\label{theorem1.1}
Let $n\equiv0\mod{4}$ with $n>4$. Then the power map $4\colon \Omega^2P^{n}(2)\to \Omega^2P^{n}(2)$ restricted to the skeleton $\sk_{4n-9}( \Omega^2P^{n}(2))$ is null homotopic.
\end{thm}

As a consequence, we get an answer on the exponents of the homotopy groups up to the range, where for the first case $n=4$, the homotopy groups $\pi_k(P^4(2))$ are known up to the range~\cite{Wu}.

\begin{cor}
Let $n\equiv0\mod{4}$. Then $4\cdot \pi_k(P^n(2))=0$ for $k\leq 4n-7$.\hfill $\Box$
\end{cor}

We should point out that it is unknown whether the homotopy groups of mod $2$ Moore spaces have a bounded exponent. It is known that there are infinitely many $\Z/8$-summands (in different dimensions) occurring in the homotopy groups of mod $2$ Moore spaces~\cite{Cohen-Wu}. Our result shows that the first $\Z/8$-torsion should occur in the range at least $4$ times the connectivity in the case $n\equiv0\mod{4}$.

The methodology for proving Theorem~\ref{theorem1.1} is briefly described as follows. We begin to use the Cohen groups for displaying the explicit obstructions to the $4$-th power map on the single loop space $\Omega P^n(2)$. By using the shuffle relations and the Hopf invariants on general configuration spaces, the $4$-th power on the double loop space $\Omega^2P^n(2)$ up to the range is decomposed as a composite involving the Whitehead product. After handling the reduced evaluation map, Theorem~\ref{theorem1.1} is then proved in the special case when the Whitehead square $\omega_{n-1}$ is divisible by $2$ by Theorem~\ref{theorem3.6}. With having more lemmas on the Whitehead product on $P^{4n}(2)$, Theorem~\ref{theorem1.1} is finally proved by Theorem~\ref{theorem5.3}. Here, the key lemma (Lemma~\ref{lemma4.3}) is a special property for $4n$-dimensional mod $2$ Moore spaces, which is hinted by Mark Mahowald's result~\cite{Mahowald0} that $[\iota_{4n-1}, \eta_{4n-1}]=0$.

The article is organized as follows. In section~\ref{section2}, we discuss the $4$-th power map on the single loop spaces and and double loop spaces. The reduced evaluation map on mod $2$ Moore spaces is studied in section~\ref{section3}, where Theorem~\ref{theorem3.6} is the special case of Theorem~\ref{theorem1.1} in the case when the Whitehead square is divisible by $2$. We give some lemmas in section~\ref{section4}. Theorem~\ref{theorem1.1} is proved in section~\ref{section5}, where Theorem~\ref{theorem5.3} is Theorem~\ref{theorem1.1}.

\section{The $4$-th power map on looped suspensions}\label{section2}

In this section, we display the obstructions to the $4$-th power map on $\Omega \Sigma^2 X$ and $\Omega^2\Sigma^2 X$ below four times the connectivity for spaces $\Sigma^2X$ having suspension exponent $4$.

\subsection{The obstructions to the $4$-th power map on $\Omega\Sigma^2 X$}
We use the Cohen groups~\cite{Cohen2} for computing the obstructions to the $4$-th power map. Recall that the Cohen group $K^{\Z/4}_n=K^{\Z/4}_n(x_1,\ldots,x_n)$ is combinatorially defined by the generators given by the letters $x_1,\ldots,x_n$ with following relations:
\begin{enumerate}
\item the iterated commutators $[[x_{i_1},x_{i_2}],\ldots,x_{i_t}]=1$ if $p=q$ for some $1\leq p<q\leq n$, where the commutator $[a,b]=a^{-1}b^{-1}ab$, and
\item $x_i^4=1$ for $1\leq i\leq n$.
\end{enumerate}
Let $d_i\colon K_n^{\Z/4}\to K_{n-1}^{\Z/4}$ be the group homomorphism such that $d_i(x_j)=x_j$ for $j<i$, $d_i(x_i)=1$ and $d_i(x_j)=x_{j-1}$ for $j>i$. The Cohen group $H_n^{\Z/4}$ is defined as the equalizer of $d_i\colon K_n^{\Z/4}\to K_{n-1}^{|Z/4}$ for $1\leq i\leq n$. Namely, $H_n^{\Z/4}$ is the subgroup of $K^{\Z/4}_n$ consisting of the words $w\in K_n^{\Z/4}$ with the property that $d_i(w)=d_1(w)$ for $1\leq i\leq n$. For the spaces $\Sigma^2X$ satisfying the following hypothesis:
\begin{equation}\label{equation2.1}
\textrm{ the identity map } \id_{\Sigma^2X} \textrm{ is of order } 4 \textrm{ in } [\Sigma^2X, \Sigma^2X],
\end{equation}
there is a commutative diagram of groups
\begin{equation}\label{equation2.2}
\begin{diagram}
K_n^{\Z/4}&\rTo^{e_X}&[(\Sigma X)^{\times n}, \Omega\Sigma^2X]\\
\uInto&&\uInto>{q_n^{*}}\\
H_n^{\Z/4}&\rTo^{e_X}&[J_n(\Sigma X), \Omega \Sigma^2X],\\
\end{diagram}
\end{equation}
where $J(Y)$ is the James construction with the James filtration $J_n(Y)$, the monomorphism in the right column is induced by the quotient map $q_n\colon (\Sigma X)^{\times n}\to J_n(\Sigma X)$ and the group homomorphism $e_X$ sends the letter $x_i$ to the homotopy class of the composite
$$
(\Sigma X)^{\times n} \rTo^{\pi_i} \Sigma X\rInto \Omega \Sigma^2X
$$
with $\pi_i$ the $i$-th coordinate projection for $1\leq i\leq n$. Let $\alpha_n=x_1x_2\cdots x_n\in H_n^{\Z/4}\leq K_n^{\Z/4}$. Then $e_X(\alpha_n)$ is the homotopy class of the inclusion map $J_n(\Sigma X)\to \Omega\Sigma^2X$.

We are only interested in the range below four times the connectivity. It is sufficient to only consider $\alpha_3^4$, which can be done by direct computations through the Magnus-type representation of $K^{\Z/4}_n$ into the non-commutative exterior algebra $A_n^{\Z/4}$. Here, $A_n^{\Z/4}=A_n^{\Z/4}(y_1,\ldots,y_n)$ is the quotient algebra of the tensor $T(y_1,\ldots,y_n)$ over $\Z/4$ subject to the relations
\begin{equation}\label{equation2.3}
y_{i_1}y_{i_2}\cdots y_{i_t}=0\quad \textrm{ if $p=q$ for some $1\leq p<q\leq t$.}
\end{equation}
 The representation $e\colon K_n^{\Z/4}\to A_n^{\Z/4}$ is given by $e(x_i)=1+y_i$. It is proved in~\cite{Cohen2} that this is a faithful representation of $K_n^{\Z/4}$.

Now $e(\alpha_3^4)=((1+y_1)(1+y_2)(1+y_3))^4$ in $A_3^{\Z/4}$. Note that
$$
(1+y_1)(1+y_2)(1+y_3)=1+\sigma_1+\sigma_2+\sigma_3,
$$
where $\sigma_1=y_1+y_2+y_3$, $\sigma_2=y_1y_2+y_1y_3+y_2y_3$ and $\sigma_3=y_1y_2y_3$. Let $\Delta=\sigma_1+\sigma_2+\sigma_3$. Then
$$
e(\alpha_3^4)=(1+\Delta)^4=1+4\Delta+6\Delta^2+4\Delta^3+\Delta^4=1+2\Delta^2
$$
in $A_3^{\Z/4}$ because $4\alpha=0$ for $\alpha\in A^{\Z/4}_n$ and $\Delta^4\in I^4A_3^{\Z/4}=0$, the $4$-fold product of the augmentation ideal $IA_3^{\Z/4}$. By using the property that $I^4A_3^{\Z/4}=0$, we have $\Delta^2=\sigma_1^2+\sigma_2\sigma_1+\sigma_1\sigma_2$. With taking the notation $[\alpha,\beta]=\alpha\beta-\beta\alpha$ for $\alpha,\beta$ in an algebra $A$ and using the relations~(\ref{equation2.3}), we have
$$
\begin{array}{rcl}
2\sigma_1^2&=&2(y_1+y_2+y_3)^2\\
&=&2(y_2y_1+y_3y_1+y_1y_2+y_3y_2+y_1y_3+y_2y_3)\\
&=&2([y_1,y_2]+[y_1,y_3]+[y_2,y_3]),\\
\end{array}
$$
$$
\begin{array}{rcl}
2(\sigma_1\sigma_2+\sigma_2\sigma_1)&=&2(2y_1y_2y_3+y_2y_3y_1+y_1y_3y_2+y_2y_1y_3+y_3y_1y_2)\\
&=&2(y_2y_3y_1+y_1y_3y_2+y_2y_1y_3+y_3y_1y_2)\\
&=&2(y_2([y_1,y_3])+[y_1,y_3]y_2)\\
&=&2[[y_1,y_3],y_2].\\
\end{array}
$$
(\textbf{Note.} Since we are working by modulo $4$, the sign $\pm$ on the terms can be ignored after multiplying by $2$.)
Now, by using the property~\cite[Lemma 1.4.8]{Wu2} that $e([[x_{i_1},x_{i_2}],\ldots,x_{i_t}])=1+[[y_{i_1},y_{i_2}],\ldots,y_{i_t}]$, we have
\begin{equation}\label{equation2.4}
\alpha_3^4=([x_1,x_2]^2[x_1,x_3]^2[x_2,x_3]^2)\cdot [[x_1,x_3],x_2]^2.
\end{equation}
We give a remark that the above method is valid for computing $\alpha_n^4$ for small $n$, more effective method for determining $\alpha_n^4$ for general $n$ can be seen in~\cite{MW_moore}. The geometric interpretation of formula~(\ref{equation2.4}) through the representation $e_X$ gives the following lemma.

\begin{lem}[Obstruction Lemma]\label{lemma2.1}
Let $X$ be a $CW$-complex such that $4\cdot [\id_{\Sigma^2X}]=0$ in $[\Sigma^2X,\Sigma^2X]$. Let $4|_{J_3}\colon J_3(\Sigma X)\to \Omega\Sigma^2X$ be the restriction of the power map $4\colon J(\Sigma X)\simeq \Omega\Sigma^2X\to \Omega\Sigma^2X$. Then there is a decomposition
$$
[4|_{J_3}]=\zeta_2\cdot \zeta_3
$$
in the group $[J_3(\Sigma X),\Omega\Sigma^2X]$, where $\zeta_2$ is represented by the composite
$$
J_3(\Sigma X)\rInto J(\Sigma X)\rTo^{H_2} J((\Sigma X)^{\wedge 2})\rTo^{\Omega W_2^2} \Omega \Sigma^2 X
$$
with $H_k$ the $k$-th James-Hopf invariant and $W_k$ the $k$-fold Whitehead product, and $\zeta_3$ is represented by the composite
$$
J_3(\Sigma X)\rTo^{\mathrm{pinch}} (\Sigma X)^{\wedge 3}\rTo^{\tau_{2,3}} (\Sigma X)^{\wedge 3}\rTo^{W_3^2}\Omega\Sigma^2X.
$$
with $\tau_{2,3}$ the map switching positions $2$ and $3$ in the self-smash product.\hfill $\Box$
\end{lem}

\subsection{The elimination of the obstruction $\zeta_3$.} Consider the looping homomorphism
$$
\Omega\colon [J_3(\Sigma X), \Omega \Sigma^2X]\longrightarrow [\Omega J_3(\Sigma X),\Omega^2\Sigma^2 X],\quad [f]\mapsto [\Omega f].
$$
The obstruction $\zeta_3$ can be always eliminated after looping using the shuffle relations introduced in~\cite{Wu2}. Here we give a proof by highlighting the ideas of the shuffle relations.

\begin{prop}\label{proposition2.2}
The element $\zeta_3$ lies in the kernel of the looping homomorphism defined as above. Thus, for any space $Z$ and any map $f\colon \Sigma Z\to J_3(\Sigma X)$, the composite $\zeta_3\circ f\colon \Sigma Z\to \Omega \Sigma^2X$ is null homotopic.
\end{prop}
\begin{proof}
Let $Y=\Sigma X$. Let $J(Y)\wedge J(Y)$ be filtered by
$$
\mathrm{Fil}_n(J(Y)\wedge J(Y))=\bigcup_{i+j\leq n} J_i(Y)\wedge J_j(Y).
$$
Since $Y$ is a co-$H$-space, there exists a filtration-preserving map $$\bar\psi\colon J(Y)\to J(Y)\wedge J(Y)$$ such that $\bar\psi$ is homotopic to the reduced diagonal $\bar\Delta$. Consider the following homotopy commutative diagram
\begin{diagram}
J(Y)&\lInto&J_3(Y)&\rTo^{\mathrm{pinch}}& Y^{\wedge 3}\\
\dTo>{\bar\Delta}&&\dTo^{\bar\psi_3}&&\dTo^{\mathrm{shuffle}}\\
J(Y)\wedge J(Y)&\lInto&\mathrm{Fil}_3(J(Y)\wedge J(Y))&\rTo^{\mathrm{pinch}}&(J_2(Y)/J_1(Y)\wedge Y)\vee (Y\wedge J_2(Y)/J_1(Y))\\
 &\rdDashto>{H}&   &      &\dTo>{\mathrm{proj.}}\\
 &   &J(Y^{\wedge 3})&\lInto& J_2(Y)/J_1(Y)\wedge Y=Y^{\wedge 3}\\
 & &   &\rdTo^{\Omega W_3}&\dTo>{S_3}\\
 & & & &\Omega\Sigma Y,\\
 \end{diagram}
 where $S_3$ is the Samelson product and the extension map $H$ exists by the suspension splitting of $J(Y)\wedge J(Y)$. By using the Cohen program, the composite from $J_3(Y)$ goes through the right column represents the element
 $$
 [[x_1,x_2],x_3] +[[x_2,x_1],x_3]+[[x_3,x_1],x_2]=-[[x_1,x_3],x_2]
 $$
 in the Cohen group $K_3$. The assertion follows by letting the composite from $J_3(Y)$ take the path from the right hand side with using the property that $\Omega\bar\Delta\colon \Omega J(Y)\to \Omega (J(Y)\wedge J(Y))$ is null homotopic.
\end{proof}

\subsection{The configuration spaces and the obstruction $\zeta_2$} The obstruction $\zeta_2$ is essential after looping in general. We use configuration spaces for reducing the obstruction $\zeta_2$ from the cubic range to the quadratic range. We refer C. -F. B\"odigheimer's work ~\cite{Bodigheimer} as a reference on configuration space models for mapping spaces as well as his constructions of the Hopf invariants on configuration spaces.

Let $M$ be a smooth manifold, $M_0$ a submanifold, and $X$ a pointed $CW$-complex. Let $C(M,M_0;X)$ be the configuration spaces with labels in $X$ in the sense of~\cite{Bodigheimer} with the filtration $C_n=C_n(M,M_0;X)$ induced by the configuration length. Let $D_n=D_n(M,M_0;X)$ denote $C_n/C_{n-1}$. We will use the following properties:
\begin{enumerate}
\item~\cite[Lemma, p. 178]{Bodigheimer} Let $N$ be a codimension zero submanifold $M$. Then the isotopy cofibration
$$
(N, N\cap M_0)\rTo (M, M_0) \rTo(M, N\cup M_0)
$$
induces a quasi-fibration
$$
C(N, N\cap M_0;X)\rTo C(M,M_0;X)\rTo C(M, N\cup M_0;X)
$$
provided that $(N,N\cap N_0)$ or $X$ is connected.
\item~\cite[Section 3]{Bodigheimer} Let $V=\bigvee\limits_{k=1}^\infty D_k$. There is a Cohen construction~\cite{Cohen1} as a power set map $P\colon C(M,M_0;X)\to  C(\R^\infty;V)$, which is natural on $(M,M_0)$ and $X$, inducing a stable splitting of $C(M,M_0;X)$. The Hopf invariant is given by the composite
$$
H_k\colon C(M,M_0;X)\rTo^{P} C(\R^\infty;V)\rTo^{\mathrm{proj.}} C(\R^\infty; D_k)
$$
for $k\geq1$.
\end{enumerate}
In particular, let $I=[0,1]$, there is a quasi-fibration
$$
C([0,1]\times I;X)\rTo C(([0,3], [2,3])\times I;X)\rTo C(([0,3],[0,1]\cup[2,3])\times I;X)
$$
for any path-connected $CW$-complex $X$ with
$$C(([0,3], [2,3])\times I;X)\simeq \ast \textrm{ and } C([0,1]\times I;X)\simeq \Omega^2\Sigma^2X.$$
We choose the evaluation map $$\Sigma \Omega^2\Sigma^2X\simeq \Sigma C(I^2;X)\to \Omega\Sigma^2X\simeq J(\Sigma X)$$ as the composite of
$$
\Sigma C(I^2;X)\simeq C(([0,3], \partial_{+})\times I;X)/C(I\times I;X)\rTo^{\mathrm{pinch}} C(([0,3],\partial)\times I;X),
$$
where $\partial_{+}=[2,3]$ and $\partial=[0,1]\cup [2,3]$, followed by composing the homotopy inverse of the composite
$$
J(\Sigma X)\rInto^{\simeq} C(I;\Sigma X)\rInto^{\simeq} C(([0,3],\partial)\times I;X).
$$
The evaluation map $\sigma\colon \Sigma C(I^2;X)\to J(\Sigma X)$ defined in such a way is a filtration-preserving map up to homotopy, and so its restrictions give maps
\begin{equation}\label{equation2.5}
\sigma_k\colon\Sigma C_k(I^2;X)\longrightarrow J_k(\Sigma X)
\end{equation}
inducing the \textit{reduced evaluation maps}
\begin{equation}\label{equation2.6}
\bar\sigma_k\colon \Sigma D_k(I^2;X)=\Sigma (C_k/C_{k-1})\longrightarrow (\Sigma X)^{\wedge k}=J_k(\Sigma X)/J_{k-1}(\Sigma X)
\end{equation}
for $k\geq 1$, where $\sigma_1=\bar\sigma_1\colon \Sigma C_1(I^2;X)=\Sigma X\to J_1(\Sigma X)$ is a homotopy equivalence. Moreover, by applying the naturality of the Hopf invariants on $(M,M_0)$ to the composites in the definition of the evaluation map $\sigma$, there is a homotopy commutative diagram
\begin{equation}\label{equation2.7}
\begin{diagram}
    &     & \Sigma C(I^2;X)&     &\rTo^{\Sigma H_k}& \Sigma C(\R^\infty;D_k) &\rEq  &\Sigma C(\R^\infty;D_k)\\
    &\ruInto&\dTo>{\sigma}&     &\ruInto&           &&\dDashto\\
\Sigma C_k & \rTo^{\mathrm{pinch}}&                         &\Sigma D_k&   &&&\\
\dTo>{\sigma_k} &&         &\dTo>{\bar\sigma_k}&&&&\\
        &        &J(\Sigma X)    & &\rTo^{H_k}&J((\Sigma X)^{\wedge k})&\rInto& C(\R^{\infty};(\Sigma X)^{\wedge k})\\
        &\ruInto&                    & &\ruInto&&&\\
  J_k(\Sigma X)&\rTo^{\mathrm{pinch}}&&(\Sigma X)^{\wedge k}&&&&&\\
  \end{diagram}
  \end{equation}
Let $\sk_n(Y)$ denote the $n$-th skeleton of $Y$.
\begin{thm}\label{theorem2.3}
Let $X$ be a simply connected space with the connectivity $|X|$ such that $\id_{\Sigma^2X}$ has exponent $4$ in $[\Sigma^2X,\Sigma^2X]$. Then there is a map
$$
\tilde H_2\colon \sk_{4|X|-1}(\Omega^2\Sigma^2X)\to D_2(I^2;X)
$$
such that the adjoint map of the $4$-th power $4\colon \Omega^2\Sigma^2X\to \Omega^2\Sigma^2X$ restricted to $\sk_{4|X|-1}(\Omega^2\Sigma^2X)$ is homotopic to the composite
$$
\Sigma \sk_{4|X|-1}(\Omega^2\Sigma^2X)\rTo^{\Sigma \tilde H_2} \Sigma D_2(I^2;X)\rTo^{\bar\sigma_2} (\Sigma X)^{\wedge 2}\rTo^{2\cdot S_2} \Omega \Sigma^2 X.
$$
\end{thm}
\begin{proof}
Note that $D_2(I^2;X)$ is the $(4|X|-1)$-skeleton of $C(\R^\infty; D_2)$. There is a homotopy commutative diagram
\begin{equation}\label{equation2.8}
\begin{diagram}
\sk_{4|X|-1}(\Omega^2\Sigma^2 X)&\rInto & C_3(I^2;X)&\rInto& \Omega^2\Sigma^2X\\
\dDashto>{\tilde H_2} &  &   &   &\dTo>{H_2}\\
D_2&      &\rInto  &&C(\R^{\infty};D_2)\\
\end{diagram}
\end{equation}
for some map $\tilde H_2$. This gives the map $\tilde H_2$ in the statement.

By Proposition~\ref{proposition2.2}, $\zeta_3\circ\sigma_3$ is null homotopic. So we only need to consider the obstruction $\zeta_2$. By diagram~(\ref{equation2.7}) using the property that $|(\Sigma X)^{\wedge 2}|=2(|X|+1)$, there is a homotopy commutative diagram
\begin{diagram}
\Sigma(\sk_{4|X|-1}(\Omega^2\Sigma^2 X))&\rTo^{\sigma_3|}& J_3(\Sigma X)&  &\\
\dTo>{\Sigma \tilde H_2}&   &&\rdTo>{H_2}&\\
\Sigma D_2(I^2;X)&\rTo^{\bar\sigma_2}&(\Sigma X)^{\wedge 2}&\rInto& J((\Sigma X)^{\wedge 2}).\\
\end{diagram}
The assertion then follows by Lemma~\ref{lemma2.1}.

\end{proof}

\section{The reduced evaluation map on mod $2$ Moore spaces}\label{section3}

In this section, we give some lemmas on the reduced evaluation map
$$
\bar\sigma=\bar\sigma_2\colon \Sigma D_2=\Sigma D_2(I^2;X)=\Sigma D_2(\R^2;X)\longrightarrow (\Sigma X)^{\wedge 2}
$$
in the case that $X$ is a mod $2$ Moore space. Let $P^n(2)=S^{n-1}\cup_2e^n$ be the $n$-dimensional mod $2$ Moore space.

\begin{lem}~\cite[Proposition 2.5]{Wu}\label{lemma3.1}
Let $n\geq 3$. Then the degree $2$ map $[2]\colon P^n(2)\to P^n(2)$ is homotopic to the composite $$
P^n(2)\rTo^{\mathrm{pinch}}S^n\rTo^{\eta}S^{n-1}\rInto P^n(2).
$$
Thus the degree $2$ map $[2]\colon P^n(2)^{\wedge 2}\to P^n(2)^{\wedge 2}$ is homotopic to the composite
$$
P^n(2)^{\wedge 2}\rTo^{\mathrm{pinch}}P^{2n}(2)\rTo^{\eta\wedge\id}P^{2n-1}(2)\rInto P^n(2)^{\wedge 2}.
$$
\hfill $\Box$
\end{lem}

Let $u,v$ be a basis for the mod $2$ homology $\tilde H_*(P^n(2))$ with $|u|=n-1$ and $|v|=n$. Then the mod $2$ homology $H_*(\Omega P^{n+1}(2))=T(u,v)$ with $Sq^1_*v=u$. By the work of Dyer-Lashof~\cite{Dyer-Lashof}, $\tilde H_*(D_2(\R^2;P^{n-1}(2)))$ has the following basis

\begin{diagram}
2n-1& &                  &       &             &    &  \tau(v^2) &  &\\
        & &                    &         &            &       &      &\rdTo>{Sq^1_*}&\\
2n-2&&                        &       &\tau(v)^2&     & \dTo>{Sq^2_*} &    &\tau([u,v])\\
         &&                         &     &               &\rdTo>{\beta_2}&      &&\\
2n-3   & & \tau(u)\tau(v)    &   &\dTo>{Sq^2_*}&       &\tau(u^2)& &\\
            &&        &\rdTo>{Sq^1_*}&  &&&&\\
 2n-4&  &      &     &\tau(u)^2,&&&&\\
 \end{diagram}
where the Steenrod operations follow from that on $H_*(\Omega P^{n+1}(2))$.

The reduced evaluation $\bar\sigma_*\colon \tilde H_*(D_2(\R^2,P^{n-1}(2)))\to \tilde H_*((P^n(2))^{\wedge 2})$ is given by $\bar\sigma_*(\tau(v^2))=v^2$, $\bar\sigma_*(\tau[u,v])=[u,v]$, $\bar\sigma_*(\tau(u^2))=u^2$ and sending remaining $3$ elements to zero.

Let us do cellular analysis on the homotopy of $\Sigma D_2(\R^2;P^{n-1}(2))$. The elements $\{\tau(v^2), \tau([u,v]), \tau(v)^2, \tau(u^2)\}$ has a structure of $P^{2n-2}(4)$ attached by $2$-cells through a map $P^{2n-2}(2)\to P^{2n-2}(4)$.

\begin{lem}\label{lemma3.2}
There exists a unique $4$-cell complex $C^{2n-1}$ such that mod $2$ homology $\tilde H_*(C^{2n-1})$ has a basis $\{a_{2n-3}, b_{2n-2}, c_{2n-2}, d_{2n-1}\}$ with $\beta_2(b)=a, Sq^2_*(d)=a $ and $Sq^1_*(d)=c$.
\end{lem}
\begin{proof}
Consider the short exact sequence
\begin{diagram}
2\cdot \pi_{2n-3}(P^{2n-2}(4))=\Z/2& \lOnto& [P^{2n-2}(2), P^{2n-2}(4)]&\lInto&\pi_{2n-2}(P^{2n-2}(4))=\Z/2.\\
\end{diagram}
Let $g_1$ be the map in the commutative diagram of cofibre sequences
\begin{diagram}
S^{2n-3}&\rInto& P^{2n-2}(2)&\rTo& S^{2n-2}\\
\dTo>{[2]}&&\dTo>{g_1}&&\dEq\\
S^{2n-3}&\rInto& P^{2n-2}(4)&\rTo& S^{2n-2}.\\
\end{diagram}
Then $2[g]$ is given by the composite
$$
P^{2n-2}(2)\rTo S^{2n-2}\rTo^{\eta}S^{2n-3} \rTo P^{2n-2}(2)\rTo^{g_1}P^{2n-2}(4),
$$
which is null homotopic because $g_1|S^{2n-3}$ factors through degree $[2]\colon S^{2n-3}\to S^{2n-3}$. Then
$$
[P^{2n-2}(2), P^{2n-2}(4)]=\Z/2\oplus \Z/2.
$$
Let $g_2$ be the composite
$$
P^{2n-2}(2)\rTo^{q} S^{2n-2}\rTo^{\eta} S^{2n-3}\rInto P^{2n-2}(4).
$$
Then the $3$ essential elements in $[P^{2n-2}(2), P^{2n-2}(4)]$ are given by $[g_1], [g_2]$ and $[g_1+g_2]$.

Since $[g_1]_*, [g_1+g_2]_*\colon H_{2n-2}(P^{2n-2}(2))\to H_{2n-2}(P^{2n-2}(4))$ are nonzero, $[g_2]$ is the only homotopy class as the attaching map for $C^{2n-1}$. The proof is finished.

\end{proof}

By pinching two bottom elements, we have a pinch map $\phi\colon \Sigma D_2(\R^2;P^{n-1}(2))\to C^{2n}$ with a commutative diagram
\begin{equation}\label{equation3.1}
\begin{diagram}
\Sigma D_2(\R^2;P^{n-1}(2))&\rTo^{\bar\sigma}& P^n(2)^{\wedge 2}\\
\dTo>{\phi}&&\dTo{p_2}\\
C^{2n}&\rTo^{p_3}& P^{2n}(2),\\
\end{diagram}
\end{equation}
where $p_3$ induces an epimorphism on mod $2$ homology.

Consider the commutative diagram of cofibre sequences
\begin{diagram}
P^{2n-2}(2)&\rTo^{g_2}&P^{2n-2}(4)& \rTo&C^{2n-1}&\rTo&P^{2n-1}(2)\\
\uEq&&\uInto&& \uTo>{j_1}&&\uEq\\
P^{2n-2}(2)&\rTo^{\eta\circ q}&S^{2n-3}& \rTo&\tilde C^{2n-1}&\rTo&P^{2n-2}(2)\\
\uEq&&\uTo&&\uTo>{j_2}&&\uEq\\
P^{2n-2}(2)&\rTo^{\bar\eta\circ q}&P^{2n-3}(2)& \rTo&\bar C^{2n-1}&\rTo&P^{2n-2}(2),\\
\end{diagram}
where $\bar\eta\colon S^{2n-2}\to P^{2n-3}(2)$ is a lifting of $\eta$. Observe that
\begin{enumerate}
\item  The map  $j_1$ induces a monomorphism on mod $2$ homology.
\item The homotopy cofibre of $j_1\circ j_2\colon \bar C^{2n-1}\to C^{2n-1}$ is the same as the homotopy cofibre of $P^{2n-3}(2)\to S^{2n-3}\to P^{2n-2}(4)$.
\item The mod $2$ homology $\tilde H_*(\bar C^{2n-1})$ has a basis $\{x_{2n-1}, x_{2n-2}, x_{2n-3}, x_{2n-4}\}$ with $Sq^1_*(x_{2n-1})=x_{2n-2}, Sq^1_*(x_{2n-3})=x_{2n-4}, Sq^2_*(x_{2n-1})=x_{2n-3}, Sq^2_*(x_{2n-2})=x_{2n-4}$.
\end{enumerate}

\begin{lem}\label{lemma3.3}
The homotopy cofibre of $P^{2n-3}(2)\to S^{2n-3}\to P^{2n-2}(4)$ is $P^{2n-2}(8)$.
\end{lem}
\begin{proof}
This follows from the commutative diagram of cofibre sequences
\begin{diagram}
&& P^{2n-3}&\rEq& P^{2n-3}(2)\\
& &\dTo&&\dTo\\
S^{2n-3}&\rTo^{[4]}&S^{2n-3}&\rInto& P^{2n-2}(4)\\
\dEq&&\dTo>{[2]}& &\dTo\\
S^{2n-3}&\rTo^{[8]}&S^{2n-3}&\rTo& P^{2n-2}(8).\\
\end{diagram}
\end{proof}

\begin{lem}\label{lemma3.4}
The space $\bar C^{2n-1}\simeq \Sigma^{2n-7} \CP^2\wedge \RP^2$.
\end{lem}
\begin{proof}
We only need to show that there is a unique $4$-cell complex having the same homology as $\bar C^{2n+1}$ with the same Steenrod module structure.

Consider the homotopy classed $[P^{2n-2}(2), P^{2n-3}(2)]$. There is a short exact sequence
\begin{diagram}
\pi_{2n-3}(P^{2n-3}(2))=\Z/2&\lOnto& [P^{2n-2}(2), P^{2n-3}(2)]&\lInto &\pi_{2n-2}(P^{2n-3}(2))/2.\\
\end{diagram}
There are $3$ essential homotopy classes in $[P^{2n-2}(2), P^{2n-3}(2)]$ given as follows
\begin{enumerate}
\item $\eta\wedge\id\colon P^{2n-2}(2)\to P^{2n-3}(2)$. The homology of its cofibre has the same structure as $\tilde H_*(\Sigma^{2n-7} \CP^2\wedge \RP^2)$.
\item The composite
$h_1\colon P^{2n-2}(2)\rTo^{\mathrm{pinch}}S^{2n-2} \rTo^{\bar\eta} P^{2n-3}(2)$. The reduced mod $2$ homology of the cofibre $C_{h_1}$ has a basis $$\{y_{2n-1},y_{2n-2},y_{2n-3}, y_{2n-4}\}$$ with $Sq^1_*(y_{2n-1})=y_{2n-2}, Sq^1_*(y_{2n-3})=y_{2n-4}, Sq^2_*(y_{2n-1})=y_{2n-3}$ and $Sq^2_*(y_{2n-2})=0$.  Here $Sq^2_*(y_{2n-2})=0$ because $h_1|_{S^{2n-3}}$ is null homotopic.
\item The composite $h_2\colon P^{2n-2}(2)\rTo^{\tilde\eta} S^{2n-4}\rInto P^{2n-3}(2)$, where $\tilde\eta$ is an extension of $\eta\colon S^{2n-3}\to S^{2n-4}$. The reduced mod $2$ homology of the cofibre $C_{h_2}$ has a basis $$\{z_{2n-1},z_{2n-2},z_{2n-3}, z_{2n-4}\}$$ with $Sq^1_*(z_{2n-1})=z_{2n-2}, Sq^1_*(z_{2n-3})=z_{2n-4}, Sq^2_*(y_{2n-2})=y_{2n-4}$ and $Sq^2_*(y_{2n-1})=0$. Here $Sq^2_*(y_{2n-2})=0$ because $$P^{2n-2}(2)\rTo^{h_2} P^{2n-3}(2)\to S^{2n-3}$$ is null homotopic.

\end{enumerate}
The proof is finished.

\end{proof}

The following lemma will be useful.
\begin{lem}\label{lemma3.5}
There is a commutative diagram
\begin{diagram}
C^{2n} &\rTo^{p_3}& P^{2n}(2)\\
\dTo &  &\dTo>{\eta\wedge\id}\\
P^{2n-1}(8)&\rTo^{\rho}&P^{2n-1}(2),\\
\end{diagram}
where $\rho_*\colon H_{2n-2}(P^{2n-1}(8);\Z/2)\to H_{2n-2}(P^{2n-2}(2);\Z/2)$ is an isomorphism.
\end{lem}
\begin{proof}
Consider the commutative diagram of cofibre sequences
\begin{diagram}
\bar C^{2n}&\rTo^{j_1\circ j_2}& C^{2n} & \rTo&P^{2n-1}(8)\\
\dEq&& \dTo>{\mathrm{pinch}}&&\dTo>{\rho}\\
\bar C^{2n}&\rTo& P^{2n}(2)&\rTo^{h}& P^{2n-1}(2).
\end{diagram}
By Lemma~\ref{lemma3.4}, $h\simeq\eta\wedge\id$ and hence the result.
\end{proof}
Let us first prove the special case when the Whitehead square is divisible by $2$.
\begin{thm}\label{theorem3.6}
Let $n+1\equiv0\mod{4}$ with $n+1>5$. Suppose that the Whitehead square $\omega_n$ is divisible by $2$. Then the $4$-power map $4|\colon \Omega^2P^{n+1}(2) \to \Omega^2P^{n+1}(2)$ restricted to the skeleton $\sk_{4(n-1)-1}$is null homotopic.
\end{thm}
\begin{proof}
By Theorem~\ref{theorem2.3}, it suffices to show that the composite
$$
\Sigma D_2\rTo^{\bar\sigma_2} (P^{n}(2))^{\wedge 2}\rTo^{[2]}(P^{n}(2))^{\wedge 2}\rTo^{S_2} \Omega P^{n+1}(2)
$$
is null homotopic. By Lemma~\ref{lemma3.1}, it suffices to prove that the composite
\begin{equation}\label{equation3.2}
\Sigma D_2\rTo^{\bar\sigma_2} (P^{n}(2)^{\wedge 2}\rTo^{\mathrm{pinch}}P^{2n}(2)\rTo^{\eta\wedge\id} P^{2n-1}(2)\rTo^{S_2|}\Omega P^{n+1}(2)
\end{equation}
is null homotopic.

Since $\omega_n$ is divisible by $2$,
$$
S_2|_{S^{2n-2}}\colon S^{2n-2}\to \Omega P^{n+1}(2)
$$
is null homotopic. Thus $S_2|\colon P^{2n-1}(2)\to \Omega P^{n+1}(2)$ factors through $S^{2n-1}$.

Let $\rho\colon P^{2n-1}(8)\to P^{2n-1}(2)$ be the map in Lemma~\ref{lemma3.5}. We show that there is a commutative diagram of cofibre sequences
\begin{equation}\label{equation3.3}
\begin{diagram}
S^{2n-2}&\rInto& P^{2n-1}(8)&\rTo &S^{2n-1}\\
\dEq&&\dTo>{\rho}&&\dTo>{[4]}\\
S^{2n-2}&\rTo& P^{2n-1}(2)&\rTo&S^{2n-1}.\\
\end{diagram}
\end{equation}
There is a short exact sequence
\begin{diagram}
\pi_{2n-2}(P^{2n-1}(2))=\Z/2&\lOnto&[P^{2n-1}(8), P^{2n-1}(2)]&\lInto&\pi_{2n-1}(P^{2n-1}(2))=\Z/2\\
\end{diagram}
Let $\alpha\colon P^{2n-1}(8)\to P^{2n-1}(2)$ be an extension of the inclusion map $S^{2n-2}\to P^{2n-1}(2)$. Then $2[\alpha]=0$. Thus
$$
[P^{2n-1}(8), P^{2n-1}(2)]=\Z/2\oplus \Z/2.
$$
Let $\beta$ be the composite
$$
P^{2n-1}(8)\rTo^{\mathrm{pinch}} S^{2n-1}\rTo^{\eta} S^{2n-2}\rInto P^{2n-2}(2).
$$
Then $[P^{2n-1}(8), P^{2n-1}(2)]$ is generated by $\alpha$ and $\beta$. We can make a choice of $\alpha$ from the commutative diagram
\begin{diagram}
S^{2n-2}&\rTo^{[8]}&S^{2n-2}&\rTo&P^{2n-1}(8)&\rTo& S^{2n-1}\\
\dTo>{[4]}&&\dEq&&\dTo>{\alpha}&&\dTo{[4]}\\
S^{2n-2}&\rTo^{[2]}&S^{2n-2}& \rTo&P^{2n-1}(2)&\rTo^{p_4}& S^{2n-1}.\\
\end{diagram}
Then $[\rho]=[\alpha]$ or $[\alpha+\beta]$. Since $p_{4*}[\beta]=0$. Diagram~(\ref{equation3.3}) holds.

Now the assertion follows from the commutative diagram
\begin{diagram}
C^{2n}&\rTo^{p_3}& P^{2n}(2)&    & \\
\dTo&&\dTo>{\eta\wedge\id} &&\\
P^{2n-1}(8)&\rTo^{\rho}&P^{2n-1}(2)&\rTo^{S_2|}&\Omega P^{n+1}(2)\\
\dTo &&\dTo &\ruTo>{\lambda}&\\
S^{2n-1}&\rTo^{[4]}&S^{2n-1}& &\\
\end{diagram}
together with the fact~\cite{Cohen-Wu} that $[\lambda]\in \pi_{2n}(P^{n+1}(2))$ is of order $4$ when $n+1\equiv0 \mod{4}$.

\end{proof}

\section{Some Lemmas on $P^{2n}(2)$}\label{section4}

In this section, we give some lemmas related to the Whitehead products.

\begin{lem}\label{lemma4.1}
Let $j_{2n+1}$ be the composite
$$
\RP^{2n}\rInto SO(2n+1)\rTo\Omega^{2n+1}S^{2n+1}.
$$
Then
$$
\Omega^{2n+1}([2])\circ j_{2n+1}\simeq 2\circ j_{2n+1}.
$$
\end{lem}
\begin{proof}
By ~\cite[Proposition 4.3]{Cohen}, the maps $\Omega[2], 2\colon \Omega S^{2n+1}\to \Omega S^{2n+1}$ differ by the homotopy class represented by the composite
$$
\Omega S^{2n+1}\rTo^{H_2}\Omega S^{4n+1}\rTo^{\Omega\omega_{2n+1}}\Omega S^{2n+1}.
$$
From the commutative diagram
\begin{diagram}
\RP^{k}&\rInto&\RP^{k+1}\\
\dTo&&\dTo\\
\Omega^{k+1}S^{k+1}&\rTo&\Omega^{k+2}S^{k+2},\\
\end{diagram}
there is a commutative diagram
\begin{diagram}
S^{2k+1}&\rTo^{\Sigma^{k+1}q_k} &\Sigma^{k+1}\RP^{k}&\rInto&\Sigma^{k+1}\RP^{k+1}&\rTo^{\Sigma{k+1} p_{k+1}}& S^{2k+2}\\
\dTo&&\dTo>{j_{k+1}}&&\dTo>{j_{k+2}}&&\dTo>{\theta_{k+2}}\\
\Omega^2S^{2k+3}&\rTo^{P}&S^{k+1}&\rTo^{E}&\Omega S^{k+2}&\rTo^{H}&\Omega S^{2k+3},\\
\end{diagram}
where $q_k\colon S^k\to \RP^k$ is the projection map, $p_{k+1}\colon \RP^{k+1}\to S^{k+1}$ is the pinch map,  the top row is a cofibre sequence, and the bottom row is the EHP sequence. The map
$$
\theta_{2k+2}\colon H_{2k+2}(S^{2k+2})\longrightarrow H_{2k+2}(\Omega S^{2k+3})
$$
is an isomorphism~\cite[Theorem (1.1)]{James}.  It follows that there is a commutative diagram
\begin{diagram}
\Sigma^{2n}\RP^{2n}&\rTo^{\Sigma^{2n}p_{2n}}&S^{4n}&\rTo^{\Sigma^{2n} q_{2n}}&\Sigma^{2n}\RP^{2n}\\
\dTo>{j'_{2n+1}}&&\dInto>{\theta_{4n}}&&\dTo>{j'_{2n+1}}\\
\Omega S^{2n+1}&\rTo^{H}&\Omega S^{4n+1}&\rTo^{\Omega \omega_{2n+1}}&\Omega S^{2n+1}.\\
\end{diagram}
We check that the composite
$$
\Sigma^2\RP^{2n} \rTo^{\Sigma^2 p_{2n}} S^{2n+2}\rTo^{\Sigma^2 q_{2n}}\Sigma^2\RP^{2n}
$$
is null homotopic. If so, the assertion will follow from the above commutative diagram.

Consider the Hopf map
$$
H\colon \Sigma \RP^{\infty}\wedge \RP^{\infty} \longrightarrow \Sigma\RP^{\infty}.
$$
The composite
$$
\Sigma \RP^{2n}\wedge \RP^1\rInto \Sigma\RP^{\infty}\wedge \RP^{\infty}\rTo^{H}\Sigma\RP^{\infty}
$$
maps into $\Sigma \RP^{2n+1}$ by the skeleton reasons. Let
$$
f\colon \Sigma^2\RP^{2n}=\Sigma\RP^{2n}\wedge\RP^1\longrightarrow \Sigma\RP^{2n+1}
$$
be the resulting map. Recall that the mod $2$ homology $H_*(\RP^{\infty})=\Gamma(u)$ with $|u|=1$ is the divided algebra. The Hopf map $H$ induces
$$
H_*(\Sigma (\gamma_{2n}(u)\otimes \gamma_1(u)))=\Sigma\gamma_{2n+1}(u).
$$
Thus
$$
f_*\colon H_{2n+2}(\Sigma^2\RP^{2n};\Z/2)\longrightarrow H_{2n+2}(\Sigma \RP^{2n+1};\Z/2)
$$
is an isomorphism. It follows that the pinch map $\Sigma^2 p_{2n}\colon \Sigma^2\RP^{2n}\to S^{2n+2}$ lifts to $\Sigma \RP^{2n+1}$ by $f$. From the cofibre sequence
$$
\Sigma \RP^{2n+1}\rTo^{\Sigma p_{2n+1}}S^{2n+2}\rTo^{\Sigma^2 q_{2n}}\Sigma^2\RP^{2n},
$$
we obtain that $\Sigma^2(q_{2n}\circ p_{2n})$ is null homotopic.
\end{proof}

\begin{lem}\label{lemma4.2}
There is a homotopy commutative diagram
\begin{diagram}
\Sigma^{2n+1}\RP^{2n}& & \rTo^{[2]} &&\Sigma^{2n+1}\RP^{2n}& & \rTo& &\RP^{2n}\wedge P^{2n+2}(2) &&\\
\dTo>{j'_{2n+1}}&\rdInto&&&  \dTo>{j'_{2n+1}}&\rdInto&&&\dTo>{k_{2n+2}}&\rdInto&\\
  & &\Sigma^{2n+1}\RP^{2n+2}&\rTo && &\Sigma^{2n+1}\RP^{2n+2}&\rTo &&& \RP^{2n+2}\wedge P^{2n+2}(2)\\
   &&\dTo>{j'_{2n+3}}&&&&\dTo>{j'_{2n+3}}&&&&\dTo>{k'_{2n+4}}\\
 S^{2n+1}&\rTo^{[2]}&&&S^{2n+1}&\rTo&&&P^{2n+2}(2)&&\\
 &\rdTo^{\Sigma^2}& &&&\rdTo>{\Sigma^2}&&&&\rdTo&\\
  &&\Omega^2S^{2n+1}& \rTo^{\Omega^2([2])}&&&\Omega^2S^{2n+3}&\rTo&&&\Omega^2P^{2n+4}(2)\\
\end{diagram}

\end{lem}
\begin{proof}
By Lemma~\ref{lemma4.1}, the left cube is homotopy commutative. By using the property of cofibre sequences, there is a map $k_{2n+2}\to \RP^{2n}\wedge P^{2n+2}(2)\to P^{2n+2}(2)$ so that the right cube in the diagram commutes up to homotopy.
\end{proof}

Mark Mahowald has a result~\cite[Theorem (1.1.2a)]{Mahowald0}  that $[\iota_{4n-1}, \eta_{4n-1}]=0$. For mod $2$ Moore spaces, we have the following lemma.

\begin{lem}\label{lemma4.3}
There exists a map $\delta_{4n}\colon P^{8n-2}(2)\to P^{4n}(2)$ with the following properties:
\begin{enumerate}
\item There is a homotopy commutative diagram
\begin{diagram}
P^{8n-2}(2)&\rTo^{\delta_{4n}}& P^{4n}(2)\\
\uInto&&\uInto\\
S^{8n-3}&\rTo^{\omega_{4n-1}}&S^{4n-1}.\\
\end{diagram}
\item The composite
$$
P^{8n-2}(2)\rTo^{\delta_{4n}}P^{4n}(2)\rTo^{\Sigma^2}\Omega^2P^{4n+2}(2)
$$
is null homotopic.
\item The composite
$$
P^{8n-1}(2)\rTo^{\eta\wedge\id} P^{8n-2}(2)\rTo^{\delta_{4n}}P^{4n}(2)
$$
is null homotopic.
\end{enumerate}
\end{lem}
\begin{proof}
Let $W^n_k$ be the homotopy of the homotopy fibre of the inclusion map $S^n\to \Omega^{k}S^{n+k}$, and let $W^n_k(2)$ be the homotopy fibre of the inclusion map $P^n(2)\to \Omega^kP^{n+k}(2)$. By Lemma~\ref{lemma3.2}, there is a homotopy commutative diagram,
\begin{diagram}
W^{4n}_{\infty}(2)&\rTo&P^{4n}(2)&\rTo&Q(P^{4n}(2))\\
\uTo&&\uTo>{k_{4n}}&&\uTo\\
\RP^{\infty}_{4n-1}\wedge P^{4n-1}(2)&\rTo^{\partial_{4n-2}\wedge\id}&\RP^{4n-2}\wedge P^{4n}(2)&\rTo&\RP^{\infty}\wedge P^{4n}(2)\\
\uInto&& \uInto&&\uInto\\
\Sigma^{4n-2}\RP^{\infty}_{4n-1}&\rTo^{\Sigma^{4n-2}\partial_{4n-2}}&\Sigma^{4n-1}\RP^{4n-2}&\rTo&\Sigma^{4n-1}\RP^{\infty}\\
\dTo&&\dTo>{j_{4n-1}}&&\dTo>{j'_{\infty}}\\
W^{4n-1}_{\infty}&\rTo&S^{4n-1}&\rTo&Q(S^{4n-1})\\
\end{diagram}
with a canonical morphism of fibre sequences from the bottom row to the top row for making a homotopy commutative diagram of cubic diagrams, where $\partial_k\colon \RP^{\infty}_{k+1}=\RP^{\infty}/\RP^k\to \Sigma \RP^k$ is the boundary map. Let $\delta_{4n}\colon P^{8n-2}(2)\to P^{4n}(2)$ be the composite
$$
S^{4n-1}\wedge P^{4n-1}(2)\rInto \RP^{\infty}_{4n-1}\wedge P^{4n-1}(2)\rTo^{\partial_{4n-2}\wedge\id} \RP^{4n-2}\wedge P^{4n}(2)\rTo^{k_{4n}}P^{4n}(2).
$$

From the above commutative diagram, $\delta_{4n}|\colon S^{8n-3}\to P^{4n}(2)$ is homotopic to the composite
$$
S^{8n-3}\rTo^{\omega_{4n-1}}S^{4n-1}\rInto P^{4n}(2).
$$
Moreover the composite $P^{8n-2}(2)\rTo^{\delta_{4n}}P^{4n}(2)\rTo Q(P^{4n}(2))$ is null homotopic by the construction. It follows that $\Sigma^2\delta_{4n}\colon P^{8n}(2)\to P^{4n+2}(2)$ is null homotopic by dimensional reasons.

Now we check condition (3) in the statement. Observe that the reduced mod $2$ homology of $\RP^\infty_{4n-1}$ has a basis $\{u^{k}\}$ with $k\geq 4n-1$. The Steenrod operation
$$
Sq^2(u^{4k-1})=Sq(u^{4k-4}\cdot u^3)=u^{4k-4}Sq^2(u^3)=u^{4k+1}.
$$
Thus
$$
Sq^2_*\colon H_{8n-1}(\Sigma^{4n-2}\RP^{\infty}_{4n-1})=\Z/2\longrightarrow H_{8n-3}(\Sigma^{4n-2}\RP^{\infty}_{4n-1})=\Z/2
$$
is an isomorphism. It follows that the composite
$$
S^{8n-2}\rTo^{\eta} S^{8n-3}\rInto \Sigma^{4n-2}\RP^{\infty}_{4n-1}
$$
is null homotopic. By smashing with mod $2$ Moore spaces,  the composite
$$
P^{8n-1}(2)\rTo^{\eta\wedge\id } P^{8n-2}(2)\rInto \Sigma^{4n-2}\RP^{\infty}_{4n-1}\wedge P^{4n-1}(2)
$$
is null homotopic. Condition (3) is satisfied and hence the result.

\end{proof}


\section{Proof of Theorem~\ref{theorem1.1}}\label{section5}

We use the notation $W^n_k(2)$ defined in the proof of Lemma~\ref{lemma4.3}. Consider the homotopy commutative diagram of fibre sequences
\begin{diagram}
\Omega W^{2n}_{\infty}&\rTo&E&\rTo&W^{2n}_{\infty}(2)&\rTo&W^{2n}_{\infty}\\
\dTo&&\dTo&&\dTo&&\dTo\\
\Omega S^{2n}&\rTo^{\partial}& F^{2n}\{2\}&\rTo& P^{2n}(2)&\rTo^{\mathrm{pinch}}& S^{2n}\\
\dTo&&\dTo&&\dTo&&\dTo\\
\Omega Q(S^{2n})&\rTo^2&Q(S^{2n-1})&\rTo& Q(P^{2n}(2))&\rTo&Q(S^{2n}).\\
\end{diagram}
For a space $X$, let $\{P^n(2),X\}=[P^n(2), Q(X)]$ denote the group of stable homotopy classes from $P^n(2)$ to $X$.

\begin{lem}\label{lemma5.1}
\begin{enumerate}
\item The stabilization $[P^{4n-2}(2), S^{2n}]\to\{P^{4n-2}(2),S^{2n}\}$ is an isomorphism.
\item The stablization $[P^{4n-2}(2),\Omega S^{2n}]\to [P^{4n-2}(2), \Omega Q(S^{2n})]$ is onto.
\item Let $4n\not=4,8$. Then the kernel of $[P^{8n-2}(2), S^{4n-1}]\to \{P^{8n-2}(2), S^{4n-1}\}$ is $Z/2$ generated by any map $\phi\colon P^{8n-2}(2)\to S^{4n-1}$ such that $\phi|_{S^{8n-3}}$ is the Whitehead square.
\end{enumerate}
\end{lem}
\begin{proof}
Assertions (1) and (2) follows immediately from the fact that $S^{2n}$ is the $(4n-1)$-skeleton of $\Omega S^{2n+1}$. For assertion (3), consider homotopy commutative diagram of fibre sequences
\begin{diagram}
\Omega^2 S^{8n-1}&\rEq&\Omega^2S^{8n-1}&&\\
\dTo&&\dTo>{P}&&\\
W_2^{4n-1}&\rTo&S^{4n-1}&\rTo^{\Sigma^2}&\Omega^2S^{4n+1}\\
\dTo&&\dTo>{E}&&\dEq\\
\Omega^3S^{8n+1}&\rTo^{\Omega P}&\Omega S^{4n}&\rTo^{\Omega E}& \Omega^2 S^{4n+1}\\
\dTo&&\dTo>{H}\\
\Omega S^{8n-1}&\rEq&\Omega S^{8n-1.}&&\\
\end{diagram}
Since the composite
$$
S^{8n-2}\rInto \Omega^3 S^{8n+1}\rTo^{\Omega P} \Omega S^{4n}\rTo^{H} \Omega S^{8n-1}
$$
is of degree $2$, we have
$$
\mathrm{sk}_{8n-1}(W_2^{4n-1})=P^{8n-2}(2).
$$
It follows that there is an exact sequence
$$
[P^{8n-2}(2),\Omega^3 S^{4n+1}]\to [P^{8n-2}(2), P^{8n-2}(2)]\to [P^{8n-2}(2), S^{4n-1}]\to \{P^{8n-2}(2), S^{4n-1}\}.
$$
By the proof of Lemma~\ref{lemma4.3}, $\omega_{4n-1}\circ\eta$ is null homotopic. Thus the composite
$$
P^{8n-2}(2)\rTo S^{8n-2}\rTo^{\eta} S^{8n-3}\rTo P^{8n-2}(2)\rTo S^{4n-1}
$$
is null homotopic, and so the image of $[P^{8n-2}(2), P^{8n-2}(2)]=\Z/4$ in $[P^{8n-4}(2), S^{4n-1}]$ is $\Z/2$. The proof is finished.

\end{proof}

\begin{lem}\label{lemma5.2}
There is a homotopy decomposition
$$
\Omega F^{2n}\{2\}\simeq \Omega S^{2n-1}\times \Omega S^{4n-2}\times \Omega P^{6n-3}(2)
$$
up to dimension $8n-8$.
\end{lem}
\begin{proof}
Consider the homotopy commutative diagram of fibre sequences
\begin{diagram}
\Omega S^{2n}&\rTo&S^{2n-1}&\rTo&\tau(S^{2n})&\rTo&S^{2n}\\
\uEq&&\uTo &&\uInto&&\uEq\\
\Omega S^{2n}&\rTo&F^{2n}\{2\}&\rTo& P^{2n}(2)&\rTo& S^{2n}\\
   & &\uTo&&\uTo & &\\
   &&Y&\rEq& Y,&&\\
\end{diagram}
where $\tau(S^{2n})=SO(2n+1)/SO(2n-1)=V_{2,2n+1}$ is the $2$-frame Stiefel manifold. Since the map $F^{2n}\{2\}\to S^{2n-1}$ admits a cross-section, there is a homotopy decomposition
$$
\Omega F^{2n}\{2\}\simeq \Omega Y\times \Omega S^{2n-1}.
$$
Let $f\colon \Omega S^{4n-2}\to \Omega Y$ be the extension of the inclusion of the bottom cell. Let $g\colon \Omega P^{6n-3}(2)=\Omega \Sigma L_3(P^{2n-1}(2))\to \Omega Y$ be the map in the functorial decomposition of $$\Omega \Sigma X\simeq \Omega \Sigma L_3(X)\times \mathrm{?}$$ for $2$-local spaces. Then the map
$$
\Omega S^{4n-2}\times \Omega P^{6n-3}(2)\rTo^{(f,g)}\Omega Y
$$
induces an isomorphism on homology up to dimension $8n-8$. The proof is finished.

\end{proof}

We restate Theorem~\ref{theorem1.1} as follows.

\begin{thm}\label{theorem5.3}
Let $n>1$ The power map $4\colon \Omega^2P^{4n}(2)\to \Omega^2P^{4n}(2)$ restricted to the skeleton $\sk_{4(4n-2)-1}(\Omega^2P^{4n}(2))$ is null homotopic.
\end{thm}
\begin{proof}
If the Whitehead square $\omega_{4n-1}$ is divisible by $2$, we have proved the assertion in Theorem~\ref{theorem3.6}.  Now we assume that $\omega_{4n-1}$ is not divisible by $2$. Similar to the situation in the proof of Theorem~\ref{theorem3.6},  it suffices to prove that the composite
\begin{equation}\label{equation5.1}
\Sigma D_2\rTo^{\bar\sigma_2} (P^{4n-1}(2)^{\wedge 2}\rTo^{\mathrm{pinch}}P^{8n-2}(2)\rTo^{\eta\wedge\id} P^{8n-3}(2)\rTo^{S_2|}\Omega P^{4n}(2)
\end{equation}
is null homotopic. Our proof is given by controlling  the map $$S_2|\colon P^{8n-1}(2)\to \Omega P^{4n}(2).$$

By Lemma~\ref{lemma5.2}, $\Omega F^{4n}\{2\}\simeq \Omega S^{4n-1}\times \Omega S^{8n-2}$ up to dimension $15n-5$. Thus
$$
[P^{8n-2}(2), F^{4n}\{2\}]\cong [P^{8n-2}(2), S^{4n-1}]\oplus [P^{8n-2}(2), S^{8n-2}]=[P^{8n-2}(2),S^{4n-1}]\oplus\Z/2.
$$
By (3) of Lemma~\ref{lemma3.4}, there is an exact sequence
$$
\Z/2\longrightarrow [P^{8n-2}(2), S^{4n-1}]\longrightarrow \{P^{8n-2}(2), S^{4n-1}\}.
$$
By assertions (1) and (2) of Lemma~\ref{lemma5.1},
$\mathrm{Ker}([P^{8n-2}(2), P^{4n}(2)]\to \{P^{8n-2}(2), P^{4n}(2)\})
$
is contained in
$$
\mathrm{Im}(\Z/2\oplus\Z/2\to [P^{8n-2}(2), F^{4n}\{2\}]\to [P^{8n-2}(2), P^{4n}(2)].
$$
Let $\bar \lambda_{4n}\colon P^{8n-2}(2)\to P^{4n}(2)$ be the composite
\begin{equation}\label{equation5.2}
P^{8n-2}(2)\rTo^{pinch} S^{8n-2}\rTo^{\lambda_{4n}} F^{4n}\{2\}\rTo P^{4n}(2),
\end{equation}
where $\lambda_{4n}\colon S^{8n-2}\to F^{4n}\{2\}$ inducing an isomorphism on $H_{8n-2}$. Since $\bar \lambda_{4n}|_{S^{8n-3}}$ is trivial but $\delta_{4n}|_{8n-3}$ is essential (because $\omega_{4n-1}$ is not divisible by $2$), $[\bar\lambda_{4n}]\not=[\delta_{4n}]$. Thus the elements
$\{[\bar\lambda_{4n}], [\delta_{4n}]\}$ generates a subgroup of $[P^{8n-2}(2), P^{4n}(2)]$ containing $\mathrm{Ker}([P^{8n-2}(2), P^{4n}(2)]\to \{P^{8n-2}(2), P^{4n}(2)\})
$. Since $$[S_2|]\in \mathrm{Ker}([P^{8n-2}(2), P^{4n}(2)]\to \{P^{8n-2}(2), P^{4n}(2)\}),
$$ we have $[S_2|]=[\bar\lambda_{4n}]$, $[\delta_{4n}]$ or $[\bar\lambda_{4n}+\delta_{4n}]$.

By Lemma~\ref{lemma4.3},
$$
\delta_{4n}\circ(\eta\wedge\id)\colon P^{8n-1}(2)\longrightarrow P^{4n}(2)
$$
is null homotopic.

Following  the lines in the proof of Theorem~\ref{theorem3.6}, with using the properties that   $\bar \lambda_{4n}$ factors through $S^{8n-3}$ by ~(\ref{equation5.2}) and any map $S^{8n-3}\to \Omega P^{4n}(2)$ having nontrivial Hurewicz image is of order $4$~\cite{Cohen-Wu}, the composite ~(\ref{equation5.1})
is null homotopic if $S_2|$ is replaced  by $\bar\lambda'_{4n}\colon P^{8n-3}(2)\to \Omega P^{4n}(2)$.  The proof is finished.
\end{proof}

\end{document}